\documentclass[a4paper, 12pt]{article} 

\usepackage[utf8]{inputenc}
\usepackage[T1]{fontenc}
\usepackage[a4paper, margin=2.75cm]{geometry}
\usepackage{needspace, mathscinet}

\usepackage{libertine} 
\usepackage{inconsolata} 

\usepackage{authblk, amsmath, amsthm, amssymb, enumerate, graphicx, mathtools}
\usepackage[textsize=small, textwidth=2cm, color=yellow]{todonotes}
\usepackage{thmtools, thm-restate}
\usepackage[colorlinks=true, citecolor=blue, linkcolor=blue, urlcolor=blue]{hyperref}

\usetikzlibrary{decorations.markings,arrows.meta}

\tikzset{
  midarrow/.style={
    postaction={decorate},
    decoration={markings, mark=at position .53 with {\arrow{Stealth}}}
  },
  midarrowRed/.style={
    postaction={decorate},
    decoration={markings, mark=at position .53 with {\arrow[draw=red]{Stealth}}}
  },
  midArrow/.style={
    postaction={decorate},
    decoration={markings, mark=at position .63 with {\arrow{Stealth}}}
  },
  midArrowRed/.style={
    postaction={decorate},
    decoration={markings, mark=at position .63 with {\arrow[draw=red]{Stealth}}}
  },
}

\declaretheorem[name=Theorem, numberwithin=section]{theorem}
\declaretheorem[name=Lemma, sibling=theorem]{lemma}
\declaretheorem[name=Proposition, sibling=theorem]{proposition}

\declaretheorem[name=Corollary, sibling=theorem]{corollary}

\declaretheorem[name=Claim, sibling=theorem]{claim}

\def\cqedsymbol{\ifmmode$\lrcorner$\else{\unskip\nobreak\hfil
\penalty50\hskip1em\null\nobreak\hfil$\lrcorner$
\parfillskip=0pt\finalhyphendemerits=0\endgraf}\fi}

\let\le\leqslant

\let\leq\leqslant
\let\geq\geqslant

\let\olditemize\itemize \renewcommand{\itemize}{\olditemize\itemsep0pt}

\let\OLDthebibliography\thebibliography
\renewcommand\thebibliography[1]{
  \OLDthebibliography{#1}
  \setlength{\parskip}{0pt}
  \setlength{\itemsep}{0pt plus 0.3ex}
}%
\AtBeginDocument{
   \def\MR#1{}
}%

\definecolor{CornflowerBlue}{rgb}{0.39, 0.58, 0.93}
\definecolor{Magenta}{rgb}{0.50, 0.0, 0.50}
\definecolor{AppleGreen}{rgb}{0.55, 0.71, 0.0}
\definecolor{AO}{rgb}{0.0, 0.5, 0.0}

\title{Graphs whose Eulerian trails have unique labels}
\author{Donggyu Kim}
\affil[1,3]{School of Mathematics, Georgia Institute of Technology.}
\author{Rose McCarty\thanks{Supported by the National Science Foundation under Grant No. DMS-2452111.}}
\affil{School of Mathematics and School of Computer Science, Georgia Institute of Technology.}
\author{Caleb McFarland\thanks{Supported in part by the Georgia Tech ARC-ACO Fellowship and in part by the National Science Foundation under Grant No.~DMS-2452111.}}

\begin{document}
\maketitle

\begin{abstract}
    Consider an undirected graph whose edges are labeled invertibly in a group. When does every Eulerian trail from one fixed vertex to another have the same label? We give a precise structural answer to this question. Essentially, we show that each ``$3$-connected part'' is labeled over a group which is isomorphic to $\mathbb{Z}_2^k$ for some $k$. We also show that the algorithmic problem admits a polynomial-time reduction to the word problem for the group. 
\end{abstract}

\section{Introduction}
Group-labeled graphs (also known as gain graphs or voltage graphs) provide a flexible framework for encoding various constraints on paths, cycles, and trails in graphs while allowing a natural notion of equivalence via shifting at vertices. The most common setting is when the group is $\mathbb{Z}_2$, which corresponds to parity constraints. Cayley graphs are also naturally group-labeled, and they fit into a more general framework relating cover graphs~\cite{GrossTucker77}, graph embeddings~\cite{Gross74}, and graph automorphisms~\cite{CayleyVoltageDegreeDiam}; see~\cite{GrossTucker01}. 

Group-labeled graphs also arise from coextensions of graphic matroids~\cite{GerardsGraphicMatroids, ZaslavskyMatroids1} and analogously, from isotropic systems which are ``almost graphic''~\cite{graphicIsoSystems, BouchetCircleChar}. Thus they play an important role in developing structure theory for matroid minors and vertex-minors; see~\cite{GGW14, McCartyThesis}. On a related note, the Tait graphs of knots can be viewed as signed $4$-regular planar graphs, and they can be defined on other surfaces as well; we refer the reader to~\cite{BouchetMap, EllisMonaghanMoffattTwistedDuality, Moffat12}.



In group-labeled graphs, the graph is undirected, but the two possible orientations of each edge are labeled by inverse elements in a group. Thus every edge can be traversed in either direction, but the direction affects the corresponding group element. An Eulerian trail of the graph can be viewed as traversing the edges in some particular order and orientation; thus it corresponds to a word over the group. See Figure~\ref{fig:introExample} for an example.

\begin{figure}
\centering
\begin{tikzpicture}
    \begin{scope}
        \coordinate (a) at (0,0);
        \coordinate (b) at (2.4,0);
        \coordinate (c) at (-0.4,1.5);
        \coordinate (d) at (-0.4,-1.5);

        \draw[line width=1pt, midarrow] 
            (a) --node[below] {\small $\vec{e}_1$} (b);
        \draw[line width=1pt, midarrow] 
            (c) to[bend left=30] node[below] {\small $\vec{e}_2$} (b);
        \draw[line width=1pt, midarrow] 
            (d) to[bend right=30] node[below, yshift=-0.4mm] {\small $\vec{e}_8$} (b);
        \draw[line width=1pt, midArrow] 
            (a) to[bend left=30] node[left] {\small $\vec{e}_4$} (c);
        \draw[line width=1pt, midArrow] 
            (c) to[bend left=30] node[right, xshift=0.3mm] {\small $\vec{e}_3$} (a);
        \draw[line width=1pt, midArrow] 
            (a) to[bend left=30] node[right] {\small $\vec{e}_7$} (d);
        \draw[line width=1pt, midArrow] 
            (d) to[bend left=30] node[left, xshift=0.3mm] {\small $\vec{e}_6$} (a);
        \draw[line width=1pt, midarrow] 
            (d) to[bend left=80] node[left] {\small $\vec{e}_5$} (c);

        \node at (a) [left] {\small $a$};
        \node at (b) [right] {\small $b$};

        \fill[fill=black] (a) circle (0.07);
        \fill[fill=black] (b) circle (0.07);
        \fill[fill=black] (c) circle (0.07);
        \fill[fill=black] (d) circle (0.07);
    \end{scope}
\end{tikzpicture}
\caption{The graph illustrating the Seven Bridges in K\"{o}nigsberg with one additional edge $e_5$ added. It has an Eulerian trail $\vec{e}_1 \vec{e}_2^{-1} \vec{e}_3 \vec{e}_4 \vec{e}_5^{-1} \vec{e}_6 \vec{e}_7 \vec{e}_8$ from $a$ to $b$. When we label the arcs $\vec{e}_i$ with elements $\alpha_i$ in a group, the label of this trail is $\alpha_1 \alpha_2^{-1} \alpha_3 \alpha_4 \alpha_5^{-1} \alpha_6 \alpha_7 \alpha_8$.}
\label{fig:introExample}
\end{figure}
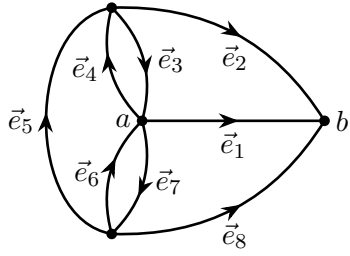

We give a structure theorem for when every Eulerian trail from one fixed vertex to another has the same label. One sufficient condition is that the group is isomorphic to $\mathbb{Z}_2^k$ for some $k \in \mathbb{N}$. Then every Eulerian trail corresponds to the same word, which is just the sum of the labels of the edges; the order does not matter since the group is abelian, and the orientation does not matter since every group element is its own inverse. We prove that this outcome is essentially the only possibility, up to shifting and some connectivity issues arising from $2$-edge-cuts. See Theorem~\ref{thm:fullTechnical} for the formal statement.

\begin{theorem}[Informal Statement]
\label{thm:informal}
The following are equivalent for any group-labeled graph $(G, \gamma)$ with vertices $a$ and $b$ such that there exists an Eulerian trail from $a$ to $b$.\begin{enumerate}
\item Every Eulerian trail from $a$ to $b$ has the same label.
\item There exists a shifting $\gamma'$ of $\gamma$ such that ``each $3$-edge-connected part of $G$'' is labeled over a group which is isomorphic to $\mathbb{Z}_2^k$ for some~$k$. 
\end{enumerate}
\end{theorem}

We give a simple proof of Theorem~\ref{thm:informal} for abelian groups in Lemma~\ref{lem:abelian}. In this case, we find that $G$ itself is labeled over a group which is isomorphic to $\mathbb{Z}_2^k$ for some~$k$. However, we observe that this does not necessarily hold for general groups; see Figure~\ref{fig:counterexample2edgeConn} for an example.

We also use Theorem~\ref{thm:informal} to efficiently solve the corresponding algorithmic problem. We assume that our graph is labeled over a finitely generated group $\Gamma$ and we have access to an oracle which solves the \emph{word problem} over $\Gamma$. That is, the oracle takes as input a word over the generators of $\Gamma$ and their inverses (viewed as formal symbols) and decides whether that word is equal to the identity element. This is famously an undecidable problem for general finitely generated groups \cite{BooneWordProblem,NovikovWordProblem}. However many groups do admit such oracles; see~\cite{BooneHigman1974} for a characterization of such groups and~\cite{wordProblemBook} for connections to the geometry of the group. 

We assume that the oracle runs in time $\phi(n)$ where $n$ is the length of the input word. Note that for all finite groups and many other interesting groups (such as word-hyperbolic groups~\cite{complexityOfDehns, hyperbolicGroupsConjugacy}), we have $\phi(n) = \mathcal{O}(n)$. We also refer the reader to~\cite{mappingClassGroup, 3ManifoldGroups} for some recent examples where the word problem can be solved in nearly linear time, and to~\cite{GrowthOfGroups} for a complexity class which captures virtually nilpotent groups.

With this setup, we can state our main algorithmic result as follows.

\begin{theorem}
\label{thm:algorithmicIntro}
    Let $(G, \gamma)$ be a connected group-labeled graph with $m$ edges so that every arc is labeled by a word of length~$1$. Then we can decide in time $\mathcal{O}(m^2)\cdot \phi(12m)$ whether every Eulerian circuit from a vertex $a$ to a vertex $b$ has the same label. Moreover, if not, then we can find two Eulerian circuits from $a$ to $b$ with different labels in time $\mathcal{O}(m^4)\cdot \phi(12m)$.
\end{theorem}

We expect the results in this paper to play an important role in understanding circuit decompositions of Eulerian graphs. For applications to vertex-minors (see~\cite{McCartyThesis, McCartyEulerian}), we are interested in understanding when a $\mathbb{Z}_2^k$-labeled graph has a \emph{circuit decomposition} (i.e., a collection of circuits which partitions the edge-set of the graph) so that all circuits begin and end at a fixed root vertex and at least $k$ of them have non-identity label. Algorithmic aspects are also of particular interest to the quantum computing community. This is because vertex-minors capture when one graph state can be ``easily'' prepared from another; see~\cite{transformingStates, LCEquiv}. McCarty~\cite{McCartyThesis} proved a min-max theorem for $\mathbb{Z}_2$-labeled graphs, but there are substantial difficulties in generalizing this theorem to even $\mathbb{Z}_2^2$. If one wants to understand the case of general groups $\Gamma$, then Theorems~\ref{thm:informal} and~\ref{thm:algorithmicIntro} are required in order to handle $2$-edge-cuts. It seems interesting that the $\mathbb{Z}_2^k$ case arises naturally from Theorem~\ref{thm:informal}.

We leave it as an open problem to understand precisely what labels of Eulerian trails are possible in a group-labeled graphs. In particular, while Theorem~\ref{thm:informal} can be used to understand whether the labels of Eulerian trails from $a$ to $b$ generate a non-trivial subgroup, we do not have any control over what this group is.

\subsection*{Related Work}

There is a long history of studying cycle and circuit decompositions of Eulerian graphs with different constraints, as well as studying constraints in group-labeled graphs. We give a short survey of these areas in order to place our results in context. 

For one, there is a line of work~\cite{DecomposingProjective, DecomposingKleinBottle} culminating with~\cite{DecomosingGraphsOnSurfaces} which shows that the following holds for any Eulerian graph $G$ drawn on a surface. For each homotopy type, there is a cycle decomposition of $G$ so that the sum of the representativities of the cycles (with respect to that homotopy type) equals the representativity of $G$. (Here, the \emph{representativity} is the minimum number of edges a curve of a given homotopy type can intersect the graph/cycle in.) Other work focuses on finding sufficient conditions to guarantee that there exists a cycle decomposition where each cycle has even length~\cite{EvenCircuitDecompCographs, EvenCircuitDecompLineOfCubic, EvenCircuitDecompPlanar, EvenCircuitDecomp} or uses at most one edge from each part in a given partition of the edge-set~\cite{ForbiddenPartsCircuitDecomp}.

M\'{a}\v{c}ajov\'{a} and \v{S}koviera~\cite{MacajovaOddEul} showed that an Eulerian graph has a circuit decomposition containing $k$ circuits of odd length if and only if it has a packing of $k$ pairwise edge-disjoint odd cycles, and the number of edges has the same parity as $k$. McCarty~\cite{McCartyEulerian} proved a precise min-max theorem for the rooted version of this problem (where every circuit in the decomposition is required to hit a fixed vertex).  If one instead wishes to find a {cycle} decomposition with many odd cycles instead of odd circuits, then things becomes much harder, but see~\cite{conjOddCyclesDecomp} for a nice conjecture. This conjecture says that if $r$ and $t$ are feasible numbers of odd cycles, then any integer between $r$ and $t$ with the same parity is also feasible.

There are also Erd\H{o}s-P\'{o}sa properties (i.e., rough dualities between packing and covering) known for edge-disjoint odd cycles in $4$-edge-connected graphs which are not necessarily Eulerian; see~\cite{packingRooted} for the rooted version and~\cite{ErdosPosa4EdgeConn} for the unrooted version. The condition of $4$-edge-connectivity is necessary due to a well-known example called the Escher Wall, which Reed~\cite{ReedEscherWall} showed is the only obstruction. We note that the only problem comes from cuts with exactly three edges, and so Eulerian graphs behave much better.

On the algorithmic side, while we do not know of prior work studying what types of Eulerian trails can exist in group-labeled graphs, there is prior work on the parities of paths and cycles. Kawarabayashi, Reed, and Wollan~\cite{kawarabayashi2011graph} gave an algorithm for the $k$-linkage problem with parity constraints. In his PhD thesis, Huynh~\cite{huynh2009linkage} generalized this to the $k$-linkage problem for group-labeled graphs, for any fixed finite group. Very recently, Liu and Yoo~\cite{liu2025disjoint} gave an algorithm for the $k$-linkage problem where the constraints come from a graph where the {undirected} edges are labeled in a finite abelian group. 

There is also a vast literature about finding vertex-disjoint paths whose labels are not the identity element in a group-labeled graph. We refer the reader to the min-max theorem in~\cite{ChudnovskyAPathsAlg, ChudnovskyAPaths}, the unified Erdős-Pósa theorem in~\cite{huynh2019unified}, and the algorithm for finding the shortest such path joining a pair of vertices in~\cite{IwataYamaguchiShortestNonzeroPath}.

\subsection*{Outline}
In Section~\ref{sec:prelim} we introduce our notation, discuss the abelian case (see Lemma~\ref{lem:abelian}), and prove a helpful lemma about how to maintain $3$-edge-connectivity (Lemma~\ref{lem:splittingMinCut}). In Section~\ref{sec:3ConnFinal} we prove Proposition~\ref{prop:3edgeConn}, which handles the $3$-edge-connected case. In Section~\ref{sec:general} we show how to obtain the ``$3$-edge-connected parts'' and prove the full structure theorem (Theorem~\ref{thm:fullTechnical}). Finally, we use this theorem to obtain the algorithm in Section~\ref{sec:algorithm} (see Theorem~\ref{thm:algorithm-general} for the decision problem and Corollary~\ref{cor:findTrails} for the algorithm which finds the trails if they exist).


\section{Preliminaries}
\label{sec:prelim}

We consider graphs to be finite and undirected, and we allow parallel edges and loops. We denote group action multiplicatively and the identity by~1. We use standard graph theory notation. In particular, if $G$ is a graph with a set of vertices $X$, then we write $\delta_G(X)$ (or simply $\delta(X)$ if the context is clear) for the set of all edges of $G$ with exactly one end in~$X$. We say that a graph is \emph{$3$-edge-connected} if it does not have a proper, non-empty set of vertices $X$ such that $|\delta(X)|\leq 2$. Thus every one vertex graph is trivially $3$-edge-connected.


Given a graph $G$, each edge $e$ has two distinct orientations, called \emph{arcs}, which we denote by $\vec{e}$ and $\vec{e}^{-1}$. Each arc $\vec{e}$ has a \emph{tail} and a \emph{head}, and we say that $\vec{e}$ is oriented \emph{from} its tail \emph{to} its head. A \emph{group-labeled graph} is a tuple $(G, \gamma)$ so that $G$ is a graph and $\gamma$ is a function which labels each arc of $G$ in a group $\Gamma$ such that for each edge $e$, we have $\gamma(\vec{e}^{-1})=\gamma(\vec{e})^{-1}$. So traversing an edge in one direction corresponds to a group element $\alpha$, and traversing it in the other direction corresponds to $\alpha^{-1}$. We call any such function $\gamma$ a \emph{group-labeling} of~$G$. We denote by $\langle G, \gamma\rangle$ the subgroup of $\Gamma$ generated by $\{\gamma(\vec{e}) : e \in E(G)\}$. 

A \emph{trail} of $G$ is a sequence of arcs so that their underlying undirected edges are all distinct, and the head of each arc is the tail of the next. Thus each trail $T=\vec{e}_1\vec{e}_2\dots\vec{e}_t$ yields a word $\gamma(\vec{e}_1)\gamma(\vec{e}_2)\dots\gamma(\vec{e}_t)$ over $\Gamma$. The corresponding group element, denoted $\gamma(T)$, is the \emph{label} of $T$. A trail is \emph{Eulerian} if it includes exactly one orientation of every edge of $G$. A trail is \emph{from} its \emph{tail}, which is the tail of its first arc, \emph{to} its \emph{head}, which is the head of its last arc. We call these vertices the \textit{ends} of $T$. The \emph{inverse} of $T$ is the trail $T^{-1}$ of arcs in the reverse order and reverse orientation; thus $\gamma(T^{-1})=\gamma(T)^{-1}$. A \emph{circuit} is a trail $T$ which only has one end.
A \emph{subcircuit} of a trail $T$ is a circuit that is a subsequence of consecutive arcs of $T$. If $T_1$ and $T_2$ are trails so that the head of $T_1$ is the tail of $T_2$, then we can concatenate them to form a new trail which we denote by~$T_1T_2$.


There is a natural equivalence relation on $\Gamma$-labeled graphs defined as follows. Given a $\Gamma$-labeled graph $(G,\gamma)$, a vertex $v \in V(G)$, and a group element $\alpha \in \Gamma$, \textit{shifting by $\alpha$ at $v$} means to append $\alpha$ at the beginning of the label of each arc with tail $v$ and to append $\alpha^{-1}$ at the end of the label of each arc with head $v$. So for instance, if $\vec{e}$ is a loop at $v$, then its new label is $\alpha\gamma(\vec{e})\alpha^{-1}$. This operation respects orientations of edges and so always results in a new $\Gamma$-labeling. We say that a group-labeling $\gamma'$ of $G$ is a \emph{shifting} of $\gamma$ if it can be obtained by a (finite) sequence of such shiftings. 

The labels of circuits behave well under shifting in the following sense. 

\begin{lemma}
\label{lem:basicProperties}
Let $(G, \gamma)$ be a group-labeled graph.  Then for any circuit $C$ of $G$ and any shifting $\gamma'$ of $\gamma$, the element $\gamma'(C)$ is a conjugate of $\gamma(C)$. In particular, $\gamma'(C)$ and $\gamma(C)$ have the same order in the group.
\end{lemma}
\begin{proof}
    Let $C= \vec{e}_1 \vec{e}_2 \dots \vec{e}_t$ be a circuit of $G$, and let $\gamma'$ be a shifting of $\gamma$. We write $v_i$ for the head of $\vec{e}_i$ for each $i$. Note that if $u$ and $v$ are different vertices, then the operations of shifting at $u$ and shifting at $v$ commute. Thus the shifting $\gamma'$ is defined by selecting one group element $\alpha_v$ for each vertex $v$ and shifting at $v$ by $\alpha_v$. We write $\alpha_{i}$ for $\alpha_{v_i}$. Then 
    \begin{align*}
    \gamma'(C) = (\alpha_t \gamma(\vec{e}_1) \alpha_1^{-1}) (\alpha_1 \gamma(\vec{e}_2) \alpha_2^{-1}) \dots (\alpha_{t-1} \gamma(\vec{e}_t) \alpha_t^{-1}) = \alpha_t \gamma(C) \alpha_t^{-1},
    \end{align*} which proves the lemma.
\end{proof}

Note that in the same manner, if $(G, \gamma)$ is a group-labeled graph with vertices $a$ and $b$ and $\gamma'$ is a shifting of $\gamma$, then every Eulerian trail from $a$ to $b$ has the same $\gamma$-label if and only if every Eulerian trail from $a$ to $b$ has the same $\gamma'$-label.

There is a well-known way of going from one group-labeled graph $(G, \gamma)$ to another one with fewer edges. For distinct arcs $\vec{e}_1$ and $\vec{e}_2$ such that the head of $\vec{e}_1$ is the tail of $\vec{e}_2$, \emph{splitting off $\vec{e}_1$ and $\vec{e}_2$} means to delete $\vec{e}_1$ and $\vec{e}_2$ and to add a new arc from the tail of $\vec{e}_1$ to the head of $\vec{e}_2$  (which is a loop if the tail of $\vec{e}_1$ and the head of $\vec{e}_2$ are the same vertex) whose label is $\gamma(\vec{e}_1)\gamma(\vec{e}_2)$. See Figure~\ref{fig:splittingOff}.

\begin{figure}
    \centering
    \begin{tikzpicture}
        \begin{scope}
            \coordinate (a) at (0,0);

            \coordinate (a1) at (45:1.7);
            \coordinate (a11) at (35:2.15);
            \coordinate (a12) at (45:2.25);
            \coordinate (a13) at (55:2.15);

            \coordinate (a2) at (135:1.7);
            \coordinate (a21) at (125:2.15);
            \coordinate (a22) at (135:2.25);
            \coordinate (a23) at (145:2.15);
            
            \coordinate (b1) at (180+30:0.6);
            \coordinate (b2) at (180+70:0.6);
            \coordinate (b3) at (180+110:0.6);
            \coordinate (b4) at (180+150:0.6);



            \draw[line width=1pt, midArrow] (a2) --node[left, yshift=-1mm] {\small $\gamma(\vec{e}_1)$} (a);
            \draw[line width=1pt, midArrow] (a) --node[right, yshift=-1mm] {\small $\gamma(\vec{e}_2)$} (a1);

            \draw[line width=1pt]
                (a1)--(a11)
                (a1)--(a12)
                (a1)--(a13)
                (a2)--(a21)
                (a2)--(a22)
                (a2)--(a23)
                (a)--(b1)
                (a)--(b2)
                (a)--(b3)
                (a)--(b4)
                ;

            \fill[fill=black] (a) circle (0.07);
            \fill[fill=black] (a1) circle (0.07);
            \fill[fill=black] (a2) circle (0.07);

            %
        \end{scope}
        \begin{scope}[xshift=3.7cm, yshift=0.35cm]
            \draw[-stealth, line width=2pt]
                (-1.25,0) -- node[above] {split off} node[below] {$\vec{e}_1$ and $\vec{e}_2$} (1.25,0);
        \end{scope}
        \begin{scope}[xshift=7.4cm]
            \coordinate (a) at (0,0);

            \coordinate (a1) at (45:1.7);
            \coordinate (a11) at (35:2.15);
            \coordinate (a12) at (45:2.25);
            \coordinate (a13) at (55:2.15);

            \coordinate (a2) at (135:1.7);
            \coordinate (a21) at (125:2.15);
            \coordinate (a22) at (135:2.25);
            \coordinate (a23) at (145:2.15);
            
            \coordinate (b1) at (180+30:0.6);
            \coordinate (b2) at (180+70:0.6);
            \coordinate (b3) at (180+110:0.6);
            \coordinate (b4) at (180+150:0.6);



            \draw[line width=1pt, midarrow] (a2) to[bend right=30] node[above, yshift=0.5mm] {\small $\gamma(\vec{e}_1)\gamma(\vec{e}_2)$} (a1);

            \draw[line width=1pt]
                (a1)--(a11)
                (a1)--(a12)
                (a1)--(a13)
                (a2)--(a21)
                (a2)--(a22)
                (a2)--(a23)
                (a)--(b1)
                (a)--(b2)
                (a)--(b3)
                (a)--(b4)
                ;

            \fill[fill=black] (a) circle (0.07);
            \fill[fill=black] (a1) circle (0.07);
            \fill[fill=black] (a2) circle (0.07);

            %
        \end{scope}
    \end{tikzpicture}
    \caption{Illustration of splitting off $\vec{e}_1$ and $\vec{e}_2$.}
    \label{fig:splittingOff}
\end{figure}
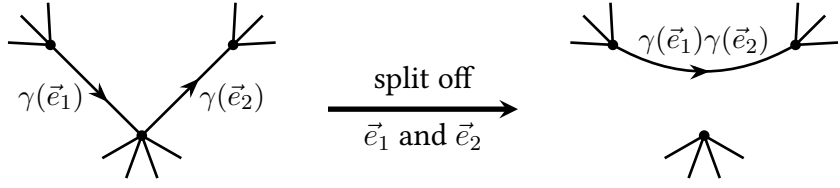

\subsection*{Abelian groups}
For abelian groups the situation is much easier, and we include a direct proof now. Then we outline how things are different for the general, non-abelian case.

\begin{lemma}
\label{lem:abelian}
Let $G$ be a graph with vertices $a$ and $b$ such that there exists an Eulerian trail from $a$ to $b$, and let $\gamma$ be a labeling of $G$ over an abelian group. Then the following are equivalent.\begin{enumerate}
    \item Every Eulerian trail of $G$ from $a$ to $b$ has the same $\gamma$-label.
    \item There exists a shifting $\gamma'$ of $\gamma$ so that $\langle G, \gamma' \rangle$ is isomorphic to $\mathbb{Z}_2^k$ for some $k$.
\end{enumerate}
\end{lemma}
\begin{proof}
    It is clear that the second condition implies the first one, so we just prove the other direction. Suppose towards a contradiction that it is false.

    First of all, suppose that $G$ has no circuit whose label has order more than $2$. Let $T$ be a spanning tree of $G$, and shift so that every edge of $T$ is labeled by the identity. Shifting does not change the order of the label of any circuit (in fact, in abelian groups, it does not change the label at all). Thus, with respect to this shifting $\gamma'$, every edge $e$ has $\gamma'(\vec{e})$ of order at most $2$. (This can be seen by considering the unique cycle of $T+e$.) So $\langle G, \gamma' \rangle$ is a finite abelian group which is generated by elements of order $2$. It follows that $\langle G, \gamma' \rangle$ is isomorphic to $\mathbb{Z}_2^k$ for some $k$, a contradiction.

    Thus there exists a circuit $C$ whose label has order greater than $2$. Choose $C$ so that the component of $G-E(C)$ which contains $a$ has as many edges as possible. Subject to this, choose $C$ so that $C$ has as many edges as possible. If no other component of $G-E(C)$ has any edges, then $b$ is in the same component as $a$ for parity reasons, and there exists an Eulerian trail $T$ of this component which goes from $a$ to $b$. Thus there are trails $T_1$ and $T_2$ so that $T=T_1 T_2$ and so that $T_1 C T_2$ is an Eulerian trail of $G$. Since $C$ has order greater than $2$, the trails $T_1 C T_2$ and $T_1 C^{-1} T_2$ have different labels, a contradiction.

    Thus there exists a component $H$ of $G-E(C)$ which has at least one edge and does not contain $a$. For parity reasons, $H$ does not contain $b$ either, and so it has an Eulerian circuit $C_H$. By cyclically changing the order of the arcs of $C$ and $C_H$, we may assume that they begin and end at the same vertex. Note that this operation does not change the label of a circuit if the group is abelian. (In general, it does not change the order of its label.) Both $C_H$ and $C_HC$ have order at most $2$ since otherwise they are better choices than $C$. Thus \begin{align*}\gamma(C_H)\gamma(C)=\gamma(C)^{-1}\gamma(C_H)^{-1}=\gamma(C)^{-1}\gamma(C_H),
    \end{align*}which contradicts the fact that $C$ has order greater than $2$. We note that this is the one place where we truly rely on the group being abelian.
\end{proof}

\begin{figure}
    \center
    \begin{tikzpicture}
        \begin{scope}
            \coordinate (0) at (2*0,0);
            \coordinate (1) at (2*1,0);
            \coordinate (2) at (2*2,0);

            \fill[fill=black] (0) circle (0.07);
            \fill[fill=black] (1) circle (0.07);
            \fill[fill=black] (2) circle (0.07);

            \draw[line width=1pt, midarrow] (0) to[bend left =45] node[above]{\small $1$} (1);
            \draw[line width=1pt, midarrow] (0) to[bend right=45] node[below]{\small $(123)$} (1);
            \draw[line width=1pt, midarrow] (1) to[bend left =45] node[above]{\small $1$}(2);
            \draw[line width=1pt, midarrow] (1) to[bend right=45] node[below]{\small $(12)$} (2);

            \node[left, xshift=-0.5mm] at (0) {\small $a=b$};
        \end{scope}
        %



    \end{tikzpicture}
    \caption{A group-labeled graph $(G, \gamma)$ whose Eulerian trails from $a$ to $b$ all have the same label $(123)(12) = (13) = (12)(132)$ in the symmetric group $\mathfrak{S}_3$, but which has no shifting $\gamma'$ so that $\langle G, \gamma' \rangle\cong \mathbb{Z}_2^k$ for some $k \in \mathbb{N}$.} 
    \label{fig:counterexample2edgeConn}
\end{figure}
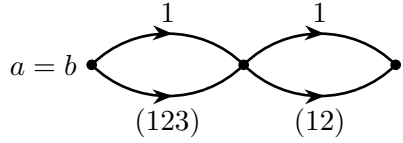

Lemma~\ref{lem:abelian} does {not} hold for all groups; it is possible that every Eulerian trail from $a$ to $b$ has the same label yet there is no shifting so that the entire generated subgroup is isomorphic to $\mathbb{Z}_2^k$ for some $k$. See Figure~\ref{fig:counterexample2edgeConn} for a concrete example. This example has no such shifting since shifting does not change the order of the label of any circuit; note that it has a circuit whose label $(123) \in \mathfrak{S}_3$ has order~$3$. However, we will prove that Lemma~\ref{lem:abelian} does hold for general groups under the additional assumption that the graph is $3$-edge-connected; see Proposition~\ref{prop:3edgeConn}. 

\subsection*{Maintaining $3$-edge-connectivity}
In order to handle the $3$-edge-connected case, we require Lemma~\ref{lem:splittingMinCut} below, which shows how to remove an edge while maintaining $3$-edge-connectivity. This lemma is essentially implied by theorems of Lov\'{a}sz~\cite{LovaszBook79} and Mader~\cite{MaderMaintainConn} about splitting off incident edges while maintaining pairwise edge-connectivity. However, as the connection is not obvious and we require a somewhat technical statement, we include a short proof anyways. We note that there are many improvements of these splitting off theorems; see for instance~\cite{Mader1995, localEdgeConn} for combinatorial improvements and~\cite{almostLinearSplittingOff, elementConnectivity, GabowSplitting, splitOffAlg} for algorithmic versions.

For convenience, we say that a \emph{valid instance} is a tuple $(G, a, b)$ such that $G$ is a $3$-edge-connected graph and $a$ and $b$ are vertices of $G$ such that there exists an Eulerian trail from $a$ to $b$. \emph{Smoothing} a vertex $v$ of a graph $G$ means that, if $v$ is incident to exactly two edges and neither of them are loops, then we delete $v$ and add a new edge whose ends are the neighbors of $v$ in $G$. (Otherwise, if $v$ does not satisfy this condition, then smoothing $v$ does nothing.) Let $(G, a, b)$ be a valid instance, and let $e$ be an edge of $G$ which is incident to $a$. We write $G_e$ for the graph obtained from $G$ by deleting $e$ and then smoothing $a$. See Figure~\ref{fig:deleteAndSmooth}. The following lemma says that there is always an edge $e$ with ends $a$ and $a'$ so that $(G_e, a', b)$ is a valid instance. Note that if $G_e \neq G-e$, then $a$ is a degree-$3$ vertex with no incident loops in $G$. So in this case, $a \neq b$, and both $a'$ and $b$ are still vertices of~$G_e$.


\begin{figure}
    \centering
    \begin{tikzpicture}
        \begin{scope}
            \coordinate (a) at (0,0);

            \coordinate (a1) at (30:1);
            \coordinate (a11) at (20:1.5);
            \coordinate (a12) at (30:1.5);
            \coordinate (a13) at (40:1.5);

            \coordinate (a2) at (150:1);
            \coordinate (a21) at (140:1.5);
            \coordinate (a22) at (150:1.5);
            \coordinate (a23) at (160:1.5);
            
            \coordinate (a3) at (270:1);
            \coordinate (a31) at (260:1.5);
            \coordinate (a32) at (270:1.5);
            \coordinate (a33) at (280:1.5);

            \draw[line width=1pt]
                (a)--(a1)
                (a)--(a2)
                (a)-- node[left, xshift=0.3mm] {\small $e$} (a3)
                (a1)--(a11)
                (a1)--(a12)
                (a1)--(a13)
                (a2)--(a21)
                (a2)--(a22)
                (a2)--(a23)
                (a3)--(a31)
                (a3)--(a32)
                (a3)--(a33);

            \fill[fill=black] (a) circle (0.07);
            \fill[fill=black] (a1) circle (0.07);
            \fill[fill=black] (a2) circle (0.07);
            \fill[fill=black] (a3) circle (0.07);

            \draw (a) node[right, yshift=-1mm] {\small $a$};
            \draw (a3) node[right, yshift=1mm] {\small $a'$};

            \node at (0,-2.0) {$G$};
        \end{scope}
        \begin{scope}[xshift=3.5cm, yshift=-0.3cm]
            \draw[-stealth, line width=2pt]
                (-1.35,0) -- node[above] {delete $e$ and} node[below] {smooth $a$} (1.35,0);
        \end{scope}
        \begin{scope}[xshift=7cm]
            \coordinate (a) at (0,0);

            \coordinate (a1) at (30:1);
            \coordinate (a11) at (20:1.5);
            \coordinate (a12) at (30:1.5);
            \coordinate (a13) at (400:1.5);

            \coordinate (a2) at (150:1);
            \coordinate (a21) at (140:1.5);
            \coordinate (a22) at (150:1.5);
            \coordinate (a23) at (160:1.5);
            
            \coordinate (a3) at (270:1);
            \coordinate (a31) at (260:1.5);
            \coordinate (a32) at (270:1.5);
            \coordinate (a33) at (280:1.5);

            \draw[line width=1pt]
                (a1) to[bend left=30] (a2)
                (a1)--(a11)
                (a1)--(a12)
                (a1)--(a13)
                (a2)--(a21)
                (a2)--(a22)
                (a2)--(a23)
                (a3)--(a31)
                (a3)--(a32)
                (a3)--(a33);

            \fill[fill=black] (a1) circle (0.07);
            \fill[fill=black] (a2) circle (0.07);
            \fill[fill=black] (a3) circle (0.07);

            \draw (a3) node[right, yshift=1mm] {\small $a'$};

            \node at (0,-2.0) {$G_e$};
        \end{scope}
    \end{tikzpicture}
    \caption{Illustrations of $G$ and $G_e$ when $G_e \neq G-e$.}
    \label{fig:deleteAndSmooth}
\end{figure}
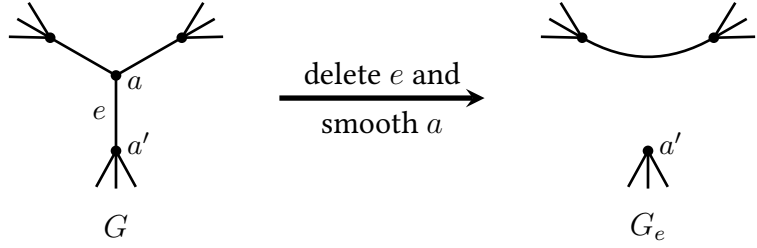

Now we prove the main lemma about how to reduce from one valid instance to another.

\begin{lemma}
\label{lem:splittingMinCut}
For any valid instance $(G, a, b)$ with at least one edge, there exists an edge $e$ with ends $a$ and $a'$ (possibly with $a=a'$) so that $(G_e, a', b)$ is a valid instance.
\end{lemma}
\begin{proof}
Going for a contradiction, suppose otherwise. First we prove the following claim.

\begin{claim}
\label{claim:EulTrailExists}
There is no loop at $a$, the vertices $a$ and $b$ are different, and for any edge $e$ with ends $a$ and $a'$, the graph $G_e$ has an Eulerian trail from $a'$ to~$b$.
\end{claim}
\begin{proof}
    It is well-known that a graph $G$ has an Eulerian trail from a vertex $a$ to a vertex $b$ if and only if $G$ is connected, every vertex besides $a$ and $b$ has even degree, and $a$ and $b$ have odd degree if they are different. (If they are the same, then this vertex has even degree by the Handshaking Lemma.) Note that for any edge $e$ incident to $a$, the graph $G-e$ and thus also the graph $G_e$ is connected since $G$ is $3$-edge-connected. Also notice that the parity conditions we just described hold for $(G_e, a', b)$ since they hold for $(G, a, b)$. Thus $G_e$ has an Eulerian trail from $a'$ to~$b$ for any edge $e$ with ends $a$ and $a'$.

    It follows that there is no loop at $a$ since otherwise we can delete this edge and retain $3$-edge-connectivity. Similarly, if $a=b$, then $G$ is Eulerian and so every cut is even. Thus $G$ is in fact $4$-edge-connected and we can take $e$ to be any edge incident to $a$, a contradiction.
\end{proof}

By Claim~\ref{claim:EulTrailExists} and since $(G,a,b)$ is a counterexample, there is no edge $e$ incident to $a$ so that $G_e$ is $3$-edge-connected. Now we prove one final claim. When we omit subscripts, we are referring to the graph $G$. So, given a set $X \subseteq V(G)$, we write $\delta(X)$ for $\delta_G(X)$.

\begin{claim}
\label{claim:edgeCuts}
For any edge $e$ incident to $a$, there exists a set of vertices $X$ which contains $a$ and at least one additional vertex, does not contain $b$, and has $|\delta(X)|=3$ and $e \in \delta(X)$.
\end{claim}
\begin{proof}
If $a$ has degree larger than $3$, then smoothing $a$ in $G-e$ does nothing and thus $G-e$ has a $2$-edge-cut. In $G$ this yields a $3$-edge-cut containing $e$. Since $G$ has an Eulerian trail from $a$ to $b$, every odd cut separates $a$ and $b$. Finally, the side containing $a$ has at least two vertices since $a$ has degree larger than $3$ by assumption.

If $a$ has degree exactly $3$, then we can find the desired set $X$ by taking a small cut in $G_e$ and adding $a$ to the side that $a$ has more neighbors in. 
\end{proof}

Now, since $G$ is $3$-edge-connected, there exist three pairwise edge-disjoint paths between $a$ and $b$. By Claim~\ref{claim:edgeCuts}, all of the neighbors of $a$ must be on one of these paths. Thus $a$ has degree exactly $3$ in $G$.

Among all edges $e$ which are incident to $a$ and all sets $X$ which satisfies the conditions of Claim~\ref{claim:edgeCuts} for $e$, choose $e$ and $X$ such that $|X|$ is minimum. Note that $X$ must contain both ends of some edge $f$ which is incident to $a$. Let $Y$ be a set of vertices which satisfies the conditions of Claim~\ref{claim:edgeCuts} for $f$.

By submodularity, $|\delta(X \cup Y)|+|\delta(X \cap Y)| \leq |\delta(Y)|+|\delta(X)|=6$. Note that $a \in X\cap Y$ and $b \notin X \cup Y$. So since $G$ is $3$-edge-connected, we actually have that $|\delta(X \cup Y)|=|\delta(X \cap Y)| = 3$. Since we chose $e$ and $X$ instead of $f$ and $X \cap Y$, we find that $a$ is the only vertex in $X \cap Y$. Note that both $X\setminus Y$ and $Y\setminus X$ are non-empty. Since $a$ has degree three, there exists one of $X\setminus Y$ and $Y\setminus X$ which has at most one edge to $a$. We suppose that $X\setminus Y$ has at most one edge to $a$; the other case is very similar. Since $G$ is $3$-edge-connected, $|\delta(X\setminus Y)|\geq 3$. The two of these edges that do not go to $a$ contribute to $\delta(X)$. However, the two edges from $a$ to $V(G)\setminus X$ also contribute to $\delta(X)$, a contradiction since $|\delta(X)|=3$. This completes the proof.
\end{proof}

\section{The 3-edge-connected case}
\label{sec:3ConnFinal}

In this section we completely take care of the $3$-edge-connected case using Lemma~\ref{lem:splittingMinCut}. For convenience, given a valid instance $(G, a,b)$, we say that a circuit $L$ of $G$ is \emph{feasible} if there exists an Eulerian trail from $a$ to $b$ which contains $L$ as a subcircuit. 

Suppose that $\gamma$ is a group-labeling of $G$ so that $(G,a,b)$ has a feasible circuit $L$ whose label has order greater than $2$. We claim that then there exist Eulerian trails from $a$ to $b$ of different labels. This is because, if $T_1$ and $T_2$ are trails such that $T_1 L T_2$ is an Eulerian trail from $a$ to $b$ (possibly with $T_1$ and/or $T_2$ being empty), then $\gamma(T_1 L T_2)\neq \gamma(T_1 L^{-1} T_2)$ since $\gamma(L)\neq \gamma(L^{-1})$. In this section we show that either there is a feasible circuit whose label has order greater than $2$, or there is a shifting $\gamma'$ of $\gamma$ such that $\langle G, \gamma'\rangle \cong \mathbb{Z}_2^k$ for some $k\in \mathbb{N}$. 

Our strategy is to maintain a stronger inductive hypothesis which essentially says that if $\langle G, \gamma'\rangle \cong \mathbb{Z}_2^k$, then every edge of $G$ is in feasible circuits whose labels generate $\mathbb{Z}_2^k$.


\begin{proposition}
\label{prop:3edgeConn}
Let $(G, a,b)$ be a valid instance, and let $\gamma$ be a group-labeling of $G$. Then exactly one of the following holds.\begin{enumerate}
    \item There exists a feasible circuit $L$ with $\gamma(L)\neq \gamma(L^{-1})$.
    \item There exists a shifting $\gamma'$ of $\gamma$ so that $\langle G, \gamma'\rangle\cong \mathbb{Z}_2^k$ for some $k\in \mathbb{N}$. Moreover, for any $e\in E(G)$, the set $\{\gamma'(L): L$ is a feasible circuit that contains $\vec{e}\}$ generates $\langle G, \gamma'\rangle$.
\end{enumerate}
\end{proposition}
\begin{proof}
    It is clear that both of the conditions cannot hold, so we just prove that at least one of the conditions holds. We prove this by induction on the number of edges in $G$. We may assume that $G$ has an edge as otherwise the second outcome trivially holds.

    By Lemma~\ref{lem:splittingMinCut}, there exists an edge $e \in E(G)$ with ends $a$ and $a'$ such that $(G_e, a', b)$ is a valid instance. If $a$ is incident to exactly two edges in $G-e$, then we write $e_1$ and $e_2$ for those two incident edges and $e_{1,2}$ for the new edge of $G_e$. We may assume that $\vec{e}_1$ is oriented towards $a$, $\vec{e}_2$ is oriented away from $a$, and $\vec{e}_{1,2}$ is oriented from the tail of $\vec{e}_1$ to the head of $\vec{e}_2$. Finally, let $\gamma_e$ denote the group-labeling of $G_e$ corresponding to $\gamma$. (First we forget about the label of $e$. Then, if $G_e \neq G-e$, we label $\vec{e}_{1,2}$ by $\gamma(\vec{e}_1)\gamma(\vec{e}_2)$.) 

    Now we apply induction to the valid instance $(G_e, a', b)$ and labeling $\gamma_e$. If the first outcome holds, then $(G, a,b)$ also has a feasible circuit $L$ so that $\gamma(L)$ has order greater than~$2$. Thus we may assume that there exists a shifting $\gamma_e'$ of $\gamma_e$ such that the second outcome holds for $(G_e, a', b)$. For convenience, we denote $\langle G_e, \gamma'_e\rangle$ by $\Gamma_e$. Let $k$ be the integer such that $\Gamma_e \cong \mathbb{Z}_2^k$. 
    
    We now define a corresponding shifting $\gamma'$ of $\gamma$ in~$G$. First perform the same shiftings (by group elements at vertices) as used to obtain $\gamma_e'$ from $\gamma_e$. Finally, if $G_e\neq G-e$, then perform one final shifting at $a$ in order to make $\vec{e}_1$ labeled by the identity.
    

    Note that every orientation of an edge in $E(G-e)$, except for possibly $e_1$ and $e_2$ if they exist, has the same label according to $\gamma'$ and $\gamma_e'$. Moreover, if $e_1$ and $e_2$ do exist, then $\gamma'(\vec{e}_1)$ is the identity, and $\gamma'(\vec{e}_2)=\gamma_e'(\vec{e}_{1,2})$. Thus $\langle G - e, \gamma'\rangle = \Gamma_e$, which is isomorphic to $\mathbb{Z}_2^k$. Without loss of generality we assume that $\vec{e}$ is oriented away from $a$, and we write $\beta$ for $\gamma'(\vec{e})$. (If $e$ is a loop, then we orient $\vec{e}$ arbitrarily.) Note that $\langle G, \gamma'\rangle$ is the group generated by $\Gamma_e \cup \{\beta\}$.

    We first show that the set $\{\gamma'(L) : L$ is a feasible circuit that contains $\vec{e}\}$ generates $\langle G, \gamma'\rangle$. We write $\mathcal{L}$ for the collection of all feasible circuits of $(G, a, b)$ whose first arc is $\vec{e}$. 

    \begin{claim}
    \label{clm:ecircuit}
        The set $\{\gamma'(L): L \in \mathcal{L}\}$ generates $\langle G, \gamma'\rangle$.
    \end{claim}
    \begin{proof}
        Because $\Gamma_e \cup \{\beta\}$ generates $\langle G, \gamma'\rangle$, it suffices to show every element in $\Gamma_e \cup \{\beta\}$ is generated by $\{\gamma'(L): L \in \mathcal{L}\}$.
    
        If $e_1$ and $e_2$ exist, then we apply the inductive statement to the edge $f = e_{1,2}$. Otherwise we apply the inductive statement to an edge $f\neq e$ which is incident to $a$. (We may assume that such an edge exists since if $e$ is the only edge of $G$, then it is a loop at the vertex $a=b$, and the claim holds.) Thus, by the inductive statement, the group generated by $\{\gamma_e'(L): L$ is a feasible circuit of $(G_e, a', b)$ that contains $\vec{f}\}$ is $\Gamma_e$. In both cases, it follows that $\Gamma_e$ is generated by the $\gamma'$-labels of the circuits $L$ of $G$ such that there exist trails $T_1, T_2$ so that $T_1 L T_2$ is an Eulerian trail from $a$ to $b$ in $G$, the first edge of $T_1$ is $\vec{e}$, and $L$ contains at least one edge that is incident to $a$. 

        Consider such trails $T_1, L, T_2$ as above. Let $L_1$ and $L_2$ be subtrails of $L$ (possibly with one of $L_1, L_2$ being empty) such that $L=L_1 L_2$ and $a$ is the last vertex of $L_1$. Then both $T_1 L_1$ and $T_1 L_2^{-1}$ are in $\mathcal{L}$. Note that \begin{align*}
        (\gamma'(T_1 L_1))^{-1}\gamma'(T_1 L_2^{-1})&=\gamma'(L_1)^{-1}\gamma'(T_1)^{-1}\gamma'(T_1)\gamma'(L_2)^{-1}\\&=\gamma'(L_1)^{-1}\gamma'(L_2)^{-1}\\&=\gamma'(L_1)\gamma'(L_2),
        \end{align*}where the last line holds because $\gamma'(L_1)$ and $\gamma'(L_2)$ have order at most~$2$. (To see this, recall that $\vec{e}$ is in $T_1$ and the other arcs generate $\Gamma_e$, which is isomorphic to $\mathbb{Z}_2^k$.) Since $\gamma'(L_1)\gamma'(L_2)= \gamma'(L)$, it follows that the $\gamma'$-labels of the circuits in $\mathcal{L}$ generate $\Gamma_e$.
    
        Moreover, for any $L \in \mathcal{L}$, there exists $\alpha \in \Gamma_e$ so that $\gamma'(L)=\beta\alpha$. Since we have shown that $\alpha$ is generated by the $\gamma'$-labels of the circuits in $\mathcal{L}$, it follows that $\beta$ is generated by them as well.
    \end{proof}

    Since every circuit in $\mathcal{L}$ is feasible for $(G, a, b)$, Claim~\ref{clm:ecircuit} implies that $\{\gamma'(L): L$ is a feasible circuit of $(G, a, b)$ that contains $\vec{e}\}$ generates $\langle G, \gamma'\rangle$.

    We now show that $\langle G, \gamma'\rangle$ is isomorphic to $\mathbb{Z}_2^{k'}$ for some $k' \in \mathbb{N}$. Clearly this holds if $\beta \in \Gamma_e$, so assume otherwise. Now consider the $\gamma'$-label of a circuit $L \in \mathcal{L}$. It must be of the form $\beta\alpha$ for some $\alpha \in \Gamma_e$. If any of these elements $\beta\alpha$ has order greater than~$2$, then the first outcome of the lemma holds, and we are done. So we may assume that each of these elements has order at most~$2$. Then any two of these elements $\beta\alpha$ and $\beta\alpha'$ commute since
        \begin{align*}
        (\beta\alpha)(\beta\alpha') &= (\alpha\beta^{-1})(\beta\alpha')
        = \alpha\alpha'
        =\alpha'\alpha
        = (\alpha'\beta^{-1})(\beta\alpha)
        = (\beta\alpha')(\beta\alpha).
        \end{align*}
    Any group which is generated by a finite number of commuting elements of order $2$ is isomorphic to $\mathbb{Z}_2^{k'}$ for some $k' \in \mathbb{N}$; thus it follows that $\langle G, \gamma'\rangle$ is isomorphic to~$\mathbb{Z}_2^{k+1}$.

    We now show that for any edge $f \in E(G - e)$, the set $\{\gamma'(L) : L$ is a feasible circuit of $(G,a,b)$ that contains $\vec{f}\}$ generates $\langle G,\gamma'\rangle$. Fix $f \in E(G - e)$, and let $\Gamma^f$ denote the group generated by the labels of feasible circuits of $(G,a,b)$ containing $\vec{f}$. Because each feasible circuit of $(G_e, a', b)$ ``lifts'' to a feasible circuit of $(G, a, b)$ of the same label, $\Gamma^f \supseteq \Gamma_e$. Thus we may assume that $\beta \not\in \Gamma_e$, and it suffices to show that $\beta \in \Gamma^f$.

    It then suffices to show that there exists a feasible circuit which contains both $e$ and $f$ as underlying undirected edges. Since the group is isomorphic to $\mathbb{Z}_2^{k+1}$, the order and orientation do not matter. Once we show that such a feasible circuit exists, it follows that $\beta$ is also generated by the set $\{\gamma'(L): L$ is a feasible circuit that contains $\vec{f}\}$, as we already know that the labels of all other edges besides $e$ are generated.
    
    Now we argue that such a feasible circuit exists. Since $G$ is $2$-edge-connected, every pair of edges of $G$ is in a common circuit by Menger's Theorem for edge-disjoint paths. Let $L$ be an edge-maximal circuit of $G$ which contains both $e$ and $f$. We claim that $L$ is feasible. By edge-maximality, every component of $G-E(L)$ which contains an edge has a vertex of odd degree. Thus, if $a=b$, then $L$ is an Eulerian circuit and we are done. If $a \neq b$, then $a$ and $b$ are in the same component of $G-E(L)$, and every other component is just an isolated vertex. Let $T$ be an Eulerian trail from $a$ to $b$ within the non-trivial component. Let $v$ be a vertex along both $T$ and $L$, and let $T = T_1T_2$ such that the head of $T_1$ equals the tail of $T_2$ equals $v$ (possibly with one of $T_1,T_2$ being empty). Then $T_1LT_2$ is an Eulerian trail of $G$, thus showing that $L$ is feasible, as desired. This completes the proof of Proposition~\ref{prop:3edgeConn}.
\end{proof}

We note Proposition~\ref{prop:3edgeConn} implies that if $G$ is a 3-edge-connected Eulerian graph with $a=b$ such that every Eulerian trail from $a$ to $b$ has the same label, then the same is true for every choice of $a=b$ (and in fact there is the same label for every choice). This does not hold for general Eulerian graphs: consider changing $a=b$ to the center vertex in Figure~\ref{fig:counterexample2edgeConn}.

\section{The general case}
\label{sec:general}

In this section we show how to reduce the general case of graphs that are not necessarily $3$-edge-connected to Proposition~\ref{prop:3edgeConn}. The main theorem is stated as Theorem~\ref{thm:fullTechnical}.

To state the full theorem, we first require some definitions. A \emph{core} (also called a \emph{$3$-core}) of a graph $G$ is a maximal set of vertices of $G$ whose pairwise edge-connectivity is at least $3$. It is well-known that the cores of a graph are pairwise-disjoint sets whose union is $V(G)$. Cores have been used for many different applications, for instance to study the structure of graphs with a forbidden immersion~\cite{WollanImmersions}. As noted in~\cite{Devos2013note}, this approach is equivalent to looking at Gomory-Hu trees, except that we do not care about the precise capacities of edges in the tree, just whether the capacities are at least~$3$. 

We use the following lemma to obtain a valid instance from each core.

\begin{lemma}
\label{lem:aboutCores}
For any graph $G$ and any core $X$ of $G$, every component of $G-X$ has at most two edges to $X$. Moreover, if $G$ has at least two cores and is connected, then for any vertex $x$ of $G$, there exists a core $X \subseteq V(G)\setminus \{x\}$ so that $G-X$ is connected.
\end{lemma}
\begin{proof}
	First of all, suppose towards a contradiction that some component $H$ of $G-X$ has at least three edges to $X$. By considering a minimum subgraph of $H$ which connects the ends of these edges, we can see that there exists a vertex $u \in V(H)$ so that there are three edge-disjoint trails from $u$ to $X$. However, since $u$ is not in the core $X$, there exists a $2$-edge-cut that separates $u$ from some vertex in $X$, and thus from all of $X$, a contradiction.
	
	Now, suppose that $G$ has at least two cores and is connected. Choose a core $X \subseteq V(G)\setminus \{x\}$ so that the component of $G-X$ containing $x$ has as many vertices as possible. If $G-X$ is not connected, then it has a component $H$ which does not contain $x$. Since there are at most two edges of $G$ joining $H$ to $X$, there exists a core $X'$ which is contained in $V(H)$. This core $X'$ contradicts our choice of $X$.
\end{proof}

Now, for convenience, we say that an \emph{instance} is a tuple $\mathcal{G}=(G, a, b, C)$ so that $G$ is a graph, $a$ and $b$ are vertices of $G$ (possibly with $a=b$), and $C$ is an Eulerian trail of $G$ from $a$ to $b$. We say that an instance (or valid instance)  is \emph{labeled} if it additionally contains a group-labeling $\gamma$ of the graph. Let $\mathcal{G}=(G, a, b, C)$ be an instance, possibly with a labeling $\gamma$, and let $X$ be a core of $G$. We now show how to \emph{obtain} a valid instance $(H, a', b')$ \emph{from} $X$, as well as an associated labeling $\gamma_H$ of $H$ from $\gamma$.

First of all, the graph $H$ has vertex-set $X$ and includes all edges $e$ of $G$ with both ends in $X$; we set $\gamma_H(\vec{e})=\gamma(\vec{e})$. Next, for each subtrail $T$ of $C$ whose first and last arcs have underlying edges in $\delta(X)$ and which begins and ends in $X$ but has no internal vertices in $X$, we add an edge $e$ to $H$ so that $\vec{e}$ is oriented from the tail of $T$ to the head of $T$; we set $\gamma_H(\vec{e})=\gamma(T)$. This completes the definition of $H$ and $\gamma_H$. Finally, we let $a'$ be the first vertex of $X$ which is hit by $C$, and we let $b'$ be the last vertex of $X$ which is hit by $C$. 
See Figure~\ref{fig:validInstanceObtainedFromCore} for example.

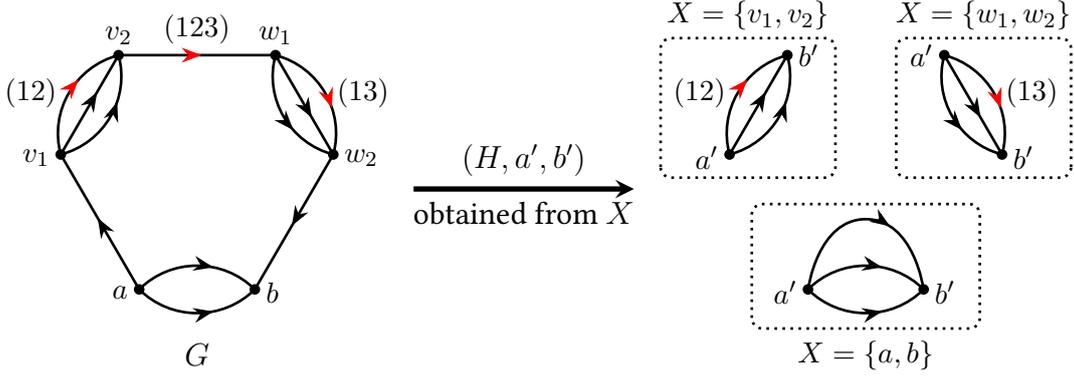
\begin{figure}
\centering
\begin{tikzpicture}
    \begin{scope}
        \coordinate (a) at (-90-25:1.8);
        \coordinate (b) at (-90+25:1.8);

        \coordinate (x1) at (150+25:1.8);
        \coordinate (x2) at (150-25:1.8);

        \coordinate (y1) at (30+25:1.8);
        \coordinate (y2) at (30-25:1.8);

        \draw[line width=1pt, midarrow] (a) to (x1);
        \draw[line width=1pt, midarrowRed] (x2) to node[above]{\small $(123)$} (y1);
        \draw[line width=1pt, midarrow] (y2) to (b);
        \draw[line width=1pt, midArrow] (a) to[bend left=45] 
        (b);
        \draw[line width=1pt, midArrow] (a) to[bend right=45] (b);
        \draw[line width=1pt, midArrowRed] (x1) to[bend left=45] node[left]{\small $(12)$} (x2);
        \draw[line width=1pt, midArrow] (x1) to[bend left=0] (x2);
        \draw[line width=1pt, midArrow] (x1) to[bend right=45] (x2);
        \draw[line width=1pt, midArrowRed] (y1) to[bend left=45] node[right]{\small $(13)$} (y2);
        \draw[line width=1pt, midArrow] (y1) to[bend left=0] (y2);
        \draw[line width=1pt, midArrow] (y1) to[bend right=45] (y2);

        \fill[fill=black] (a) circle (0.07);
        \fill[fill=black] (b) circle (0.07);
        \fill[fill=black] (x1) circle (0.07);
        \fill[fill=black] (x2) circle (0.07);
        \fill[fill=black] (y1) circle (0.07);
        \fill[fill=black] (y2) circle (0.07);

        \node at (a) [left, yshift=-0.4mm] {\small $a$};
        \node at (b) [right, yshift=-0.4mm] {\small $b$};
        \node at (x1) [left, yshift=-0.4mm] {\small $v_1$};
        \node at (x2) [above, xshift=0mm] {\small $v_2$};
        \node at (y1) [above, xshift=0mm] {\small $w_1$};
        \node at (y2) [right, yshift=-0.4mm] {\small $w_2$};

        \node at (0,-2.5) {$G$};
    \end{scope}
    \begin{scope}[xshift=4.3cm, yshift=-0.3cm]
        \draw[-stealth, line width=2pt]
            (-1.45,0) -- node[above] {$(H,a',b')$} node[below] {obtained from $X$} (1.45,0);
    \end{scope}
    \begin{scope}[xshift=8.8cm, yshift=0cm]
        \coordinate (a) at (-90-25:1.8);
        \coordinate (b) at (-90+25:1.8);
        \coordinate (virtual) at (90:-0.7);

        \draw[line width=1pt, midArrow] (a) to[bend left=45] (b);
        \draw[line width=1pt, midArrow] (a) to[bend right=45] (b);
        \draw[line width=1pt, midArrow] plot [smooth,tension=1.5] coordinates{(a) (virtual) (b)};

        \fill[fill=black] (a) circle (0.07);
        \fill[fill=black] (b) circle (0.07);

        \node at (a) [left, yshift=-0.4mm] {\small $a'$};
        \node at (b) [right, yshift=-0.4mm] {\small $b'$};

        \draw[rounded corners, line width=1pt, dotted] (-1.5,-2.15) rectangle (1.5,-0.5);

        \node at (0,-2.5) {\small $X=\{a,b\}$};
    \end{scope}
    \begin{scope}[xshift=8.8cm, yshift=0cm]
        \coordinate (x1) at (150+25:1.8);
        \coordinate (x2) at (150-25:1.8);
        \coordinate (virtual) at (-30:-0.5);

        \draw[line width=1pt, midArrowRed] (x1) to[bend left=45] node[left]{\small $(12)$} (x2);
        \draw[line width=1pt, midArrow] (x1) to[bend left=0] (x2);
        \draw[line width=1pt, midArrow] (x1) to[bend right=45] (x2);
        \node at (x1) [left, yshift=-0.4mm] {\small $a'$};
        \node at (x2) [right, xshift=0mm] {\small $b'$};

        \fill[fill=black] (x1) circle (0.07);
        \fill[fill=black] (x2) circle (0.07);


        \draw[rounded corners, line width=1pt, dotted] (-2.7,-0.15) rectangle (-0.4,1.7);

        \node at (-1.55,2.0) {\small $X=\{v_1,v_2\}$};
    \end{scope}
    \begin{scope}[xshift=8.8cm, yshift=0cm]
        \coordinate (y1) at (30+25:1.8);
        \coordinate (y2) at (30-25:1.8);
        \coordinate (virtual) at (210:-0.5);

        \draw[line width=1pt, midArrowRed] (y1) to[bend left=45] node[right]{\small $(13)$} (y2);
        \draw[line width=1pt, midArrow] (y1) to[bend left=0] (y2);
        \draw[line width=1pt, midArrow] (y1) to[bend right=45] (y2);

        \fill[fill=black] (y1) circle (0.07);
        \fill[fill=black] (y2) circle (0.07);

        \node at (y1) [left, xshift=0mm] {\small $a'$};
        \node at (y2) [right, yshift=-0.4mm] {\small $b'$};

        \draw[rounded corners, line width=1pt, dotted] (2.7,-0.15) rectangle (0.4,1.7);

        \node at (1.55,2.0) {\small $X=\{w_1,w_2\}$};
    \end{scope}
\end{tikzpicture}
\caption{An instance $(G,a,b,C)$ is illustrated on the left, where $C$ is an arbitrary Eulerian trail of $G$ from $a$ to $b$. All edges are labeled by the identity element except for the red arrowed arcs labeled by some non-identity elements in the symmetric group $\mathfrak{S}_3$. The cores of $G$ are $\{a,b\}$, $\{v_1,v_2\}$, and $\{w_1,w_2\}$. The valid instances $(H,a',b')$ obtained from these cores are illustrated on the right; we note that the precise choice of $C$ does not matter exactly because every Eulerian trail from $a$ to $b$ has the same label.}
\label{fig:validInstanceObtainedFromCore}
\end{figure}

We now prove that $(H, a', b')$ is actually a valid instance.

\begin{lemma}
\label{lem:isValid}
Let $\mathcal{G}=(G, a, b, C)$ be an instance, and let $X$ be a core of $G$. Then the tuple $(H, a', b')$ obtained from $X$ is a valid instance.
\end{lemma}
\begin{proof}
	We can see that $H$ contains an Eulerian trail from $a'$ to $b'$ by considering the maximal subtrail of $C$ whose tail is $a'$ and whose head is $b'$. 
    
    Now suppose towards a contradiction that there exists a proper, non-empty subset $Y$ of $V(H)$ such that $|\delta_H(Y)|\leq 2$. By Lemma~\ref{lem:aboutCores}, each component $K$ of $G-X$ has at most two neighbors to $X$. Let $Y'$ be the union of $Y$ and the vertex-sets of any component $K$ of $G-X$ with all of its neighbors in $Y$. Note that if $K$ is a component of $G-X$ with two distinct neighbors $x$ and $y$ so that $K$ does not contain both $a$ and $b$, then we add an edge between $x$ and $y$ to $H$. (Such a component cannot contain exactly one of $a$ and $b$ since every edge-cut separating $a$ and $b$ in $G$ is odd.) So, if no component $K$ of $G-X$ contains both $a$ and $b$ and has one neighbor in $Y$ and one neighbor in $V(H)\setminus Y$, then we have $|\delta_H(Y)|=|\delta_G(Y')|$. Moreover, $|\delta_G(Y')| \geq 3$ since the vertices in $X$ are pairwise $3$-edge-connected.

    So we may assume that there is such a component $K$. Then $|\delta_H(Y)|=|\delta_G(Y')|-1$, and $|\delta_G(Y')|$ is even since it does not contain $a$ or $b$. Then $|\delta_G(Y')| \geq 4$ since the vertices in $X$ are pairwise $3$-edge-connected. This completes the proof.
\end{proof}

We are now ready to state the full characterization. 

\begin{theorem}
\label{thm:fullTechnical}
For any instance $\mathcal{G}=(G, a,b,C)$ and labeling $\gamma$, the following are equivalent.\begin{enumerate}
\item\label{itm:cond1} Every Eulerian trail of $G$ from $a$ to $b$ has the same $\gamma$-label.

\item\label{itm:cond2} There exists a shifting $\gamma'$ of $\gamma$ so that for each core $X$, the valid instance $(H, a', b')$ obtained from $X$ and the labeling $\gamma_H$ obtained from $\gamma'$ satisfy $\langle H, \gamma_H\rangle \cong \mathbb{Z}_2^k$ for some~$k$.
\end{enumerate}
Moreover, if neither of these conditions holds, then $G$ has a circuit $L$ so that $\gamma(L) \neq \gamma(L^{-1})$ and there exists an Eulerian trail from $a$ to $b$ which has $L$ as a subcircuit.
\end{theorem}
\begin{proof}
    We prove the theorem by induction on the number of cores of $G$. For the base case, if there is only one core, then the theorem holds by Proposition~\ref{prop:3edgeConn}.

    Thus we may assume that there are at least two cores. So by Lemma~\ref{lem:aboutCores} (note that $G$ is connected since it has an Eulerian trail from $a$ to $b$), there exists a core $X$ of $G$ so that $b \notin X$ and $|\delta_G(X)|\leq 2$. Let $(H, a', b')$ and $\gamma_H$ be obtained from $X$ and $\gamma$. By Lemma~\ref{lem:isValid}, the tuple $(H, a', b')$ is a valid instance. It follows from Proposition~\ref{prop:3edgeConn} that every Eulerian trail from $a'$ to $b'$ in $H$ has the same $\gamma_H$-label if and only if there exists a shifting $\gamma_H'$ of $\gamma_H$ such that $\langle H, \gamma_H\rangle \cong \mathbb{Z}_2^k$ for some $k$. Moreover, if these conditions do not hold, then $(H, a', b')$ has a feasible circuit $L$ whose label has order greater than~$2$.
    
    We now define a new valid instance $\hat{\mathcal{G}}=(\hat{G}, \hat{a}, b, \hat{C})$ and labeling $\hat{\gamma}$ of $\hat{G}$. If $X$ contains $a$, then $|\delta(X)|=1$ and we let $\hat{a}$ be the unique neighbor of $X$. We also let $\hat{G}$ be obtained from $G$ by deleting $X$, and we let $\hat{C}$ be obtained from $C$ by deleting all arcs up to and including the arc from $\delta(X)$. If on the other hand $X$ does not contain $a$, so $|\delta(X)| = 2$, then we set $\hat{a}=a$, and we let $\hat{\mathcal{G}}$ be obtained from $\mathcal{G}$ as follows. First, let $T$ be the subtrail of $C$ whose first and last arcs are in $\delta(X)$ and which begins and ends in $V(G)\setminus X$ and has internal vertices in $X$. Then $\hat{C}$ is obtained from $C$ by replacing $T$ with a new arc from the tail of $T$ to the head of $T$. Likewise, $\hat{G}$ is obtained from $G$ by deleting $X$ and adding a new edge $e$ so that $\vec{e}$ is oriented from the tail of $T$ to the head of $T$; we set $\hat{\gamma}(e)=\gamma(T)$.

    Note that the cores of $\hat{G}$ are exactly the cores of $G$ except that $X$ is not included. (To see this, observe that every pair of vertices in $\hat{G}$ has the same edge-connectivity as in $G$.) So by induction, the two conditions are equivalent for $\hat{\mathcal{G}}$ with labeling $\hat{\gamma}$. Also, notice that every Eulerian trail of $G$ from $a$ to $b$ has the same $\gamma$-label if and only if both: 1) every Eulerian trail of $\hat{G}$ from $\hat{a}$ to $b$ has the same $\hat{\gamma}$-label and 2) every Eulerian trail of $H$ from $a'$ to $b'$ has the same $\gamma_H$-label. Also, note that each core of $\hat{G}$ yields the same valid instance with respect to $\hat{\mathcal{G}}$ as it does with respect to $\mathcal{G}$ (since the relevant subtrails of $\hat{C}$ yield subtrails of $C$ with the same ends). Likewise, shifting at vertices in a core $Y$ does not change the labeling of the valid instance of any other core besides $Y$. Finally, if any one core has a feasible circuit $L$, then $L$ is also a subcircuit of an Eulerian trail from $a$ to $b$ in $G$. The theorem follows.
\end{proof}


\section{The algorithm}
\label{sec:algorithm}

We first give a simple algorithm for testing whether every Eulerian trail from $a$ to $b$ has the same label in a 3-edge-connected group-labeled graph utilizing Proposition~\ref{prop:3edgeConn}. We then give the algorithm for general group-labeled graphs using Theorem~\ref{thm:fullTechnical}. Finally, we use this algorithm for the decision problem to actually find the Eulerian trails of different labels if they exist. Throughout this section we assume that our graph is labeled over a finitely generated group $\Gamma$ and we have access to an oracle which solves the {word problem} over $\Gamma$ in time $\phi(n)$ where $n$ is the length of the input word.

We also assume that the $\Gamma$-labeled graph $(G,\gamma)$ taken as input to the algorithms below is stored as an adjacency list, that from an edge $e$ we can access the labels of $\vec{e}$ and $\vec{e}^{-1}$, and that if $\gamma(\vec{e})$ is the word $\alpha_1\alpha_2 \ldots \alpha_t$, then $\gamma(\vec{e}^{-1})$ is the formal \emph{inverse word} $\alpha_t^{-1}\ldots \alpha_2^{-1} \alpha_1^{-1}$. For any arbitrary way of choosing an orientation $\vec{e}$ of each edge $e$, we define the \emph{total word length} of $(G, \gamma)$ to be the sum over all edges $e \in E(G)$ of the word length of~$\gamma(\vec{e})$. When we say that an algorithm \emph{returns a shifting} $\gamma'$ of $\gamma$, we mean that in addition to having the new label of every arc, we also know, for each vertex $v$, the word $\alpha_v$ which is the group element we shifted by. (Recall that any shifting can be specified in this manner.) The \emph{difference} of $\gamma'$ from $\gamma$ is the maximum length of one of these words~$\alpha_v$.


We first show how to ``normalize'' group-labelings. For a spanning tree $T$ of a connected group-labeled graph $(G, \gamma)$, a shifting $\gamma'$ is \emph{$T$-normalized} if $\gamma'(\vec{e}) = 1$ for all edges $e$ in~$T$. We prove that any two $T$-normalized shiftings are ``conjugate'' in the following sense.

\begin{lemma}
\label{lem:normalizedShifting1}
    Let $(G,\gamma)$ be a connected group-labeled graph, and let $T$ be a spanning tree of $G$. If $\gamma'$ and $\gamma''$ are $T$-normalized shiftings of $\gamma$, then there exists a group element $\alpha$ such that $\alpha \gamma'(\vec{e}) \alpha^{-1} = \gamma''(\vec{e})$ for all edge $e \in E(G)$.
\end{lemma}
\begin{proof}
    Because $\gamma'$ and $\gamma''$ are shiftings of the common function $\gamma$, we have that $\gamma''$ is a shifting of $\gamma'$. Note that for distinct vertices $u$ and $v$, it does not matter whether we first shift at $u$ or first shift at $v$. Thus we may assume that $\gamma''$ is obtained from $\gamma'$ by shifting by some element, say $\alpha_v$, at each vertex $v$. 
    
    Fix any vertex $u$, and let $\alpha := \alpha_u$. For any vertex $v$, there is a unique path $P$ in $T$ from $u$ to $v$. Because $\gamma'(\vec{e}) = \gamma''(\vec{e}) = 1$ for all edges $e$ in $P$, we deduce that $\alpha_v = \alpha$. Therefore, $\gamma''$ is obtained from $\gamma'$ by shifting by $\alpha$ at every vertex, and thus $\alpha \gamma'(\vec{e}) \alpha^{-1} = \gamma''(\vec{e})$ for all edges $e$ of $G$, as desired.
\end{proof}

Next we show how to go from an arbitrary group-labeling to a normalized one.

\begin{lemma}
\label{lem:normalizedShifting2}
    Let $(G,\gamma)$ be a connected group-labeled graph of total word length $\ell$, and let $T$ be a spanning tree of $G$. Then in time $\mathcal{O}(\ell|E(G)|)$ we can find a $T$-normalized shifting $\gamma'$ of $\gamma$ of difference at most $\ell$ from $\gamma$ such that $\langle G, \gamma' \rangle$ is a subgroup of $\langle G, \gamma \rangle$.
\end{lemma}
\begin{proof}
    We construct a $T$-normalized shifting as follows. Fix an arbitrary root of $T$, and proceed from the root to the leaves in a breadth-first fashion. We do not shift at the root vertex. For each other vertex, in order, we shift by the label of the arc to its parent so that that arc now has label 1. Because we only ever shift by elements of $\langle G, \gamma\rangle$, the resulting shifting $\gamma'$ satisfies that $\langle G, \gamma'\rangle$ is a subgroup of $\langle G, \gamma \rangle$.
    
    Each time we shift at a vertex $v$, the length of the word $\alpha_v$ which we shift by is the sum of the word lengths of arcs on the path from $v$ to the root of $T$. Thus we only ever shift by a word of length at most $\ell$. The labels on the arcs can be written in time $\mathcal{O}(\ell|E(G)|)$ considering that it may take $\mathcal{O}(\ell)$ time to write the label of a single arc.
\end{proof}

Using the above lemmas, we show that to check whether $(G,\gamma)$ has a shifting $\gamma'$ such that $\langle G, \gamma'\rangle \cong \mathbb{Z}_2^k$ for some $k$, it suffices to check an arbitrary $T$-normalized shifting.

\begin{lemma}
\label{lem:normalizedShifting3}
    Let $(G, \gamma)$ be a connected group-labeled graph, and let $T$ be a spanning tree of $G$.
    If there exists a shifting $\gamma'$ of $\gamma$ such that $\langle G, \gamma' \rangle \cong \mathbb{Z}_2^k$ for some $k\in \mathbb{N}$, then for any $T$-normalized shifting $\gamma''$ of $\gamma$, we have $\langle G, \gamma'' \rangle \cong \mathbb{Z}_2^{r}$ for some $r \in \mathbb{N}$ with $r \leq k$.
\end{lemma}
\begin{proof}
    Suppose that there is a shifting $\gamma'$ of $\gamma$ such that $\langle G, \gamma' \rangle \cong \mathbb{Z}_2^k$ for some $k\in \mathbb{N}$. By Lemma~\ref{lem:normalizedShifting2}, there is a $T$-normalized shifting $\gamma''$ of $\gamma'$ (so also of $\gamma$) such that $\langle G, \gamma'' \rangle$ is a subgroup of $\langle G, \gamma' \rangle$, which implies that $\langle G, \gamma'' \rangle \cong \mathbb{Z}_2^r$ for some $r \in \mathbb{N}$ with $r \le k$. By Lemma~\ref{lem:normalizedShifting1}, for any $T$-normalized shifting $\gamma'''$ of $\gamma$, we have $\langle G, \gamma''' \rangle \cong \langle G, \gamma'' \rangle \cong \mathbb{Z}_2^r$.
\end{proof}

Lemma~\ref{lem:normalizedShifting3} and Proposition~\ref{prop:3edgeConn} immediately give the following simple algorithm for 3-edge-connected graphs with an Eulerian trail from $a$ to $b$. By Lemma~\ref{lem:abelian}, we can also drop the assumption about $3$-edge-connectivity when the group is abelian.


\begin{proposition}
\label{prop:algorithm3-edge-conn}
    Let $(G, \gamma)$ be a $3$-edge-connected group-labeled graph with $m$ edges and total word length $\ell$. Let $a$ and $b$ be vertices of $G$ such that there exists an Eulerian trail from $a$ to $b$. Then it can be decided in $\mathcal{O}(\ell m)\cdot \phi(12\ell)$ time whether every Eulerian trail from $a$ to $b$ has the same label. Moreover, in the affirmative case, the algorithm finds a shifting $\gamma'$ of $\gamma$ of difference at most $\ell$ from $\gamma$ such that $\langle G, \gamma'\rangle$ is isomorphic to $\mathbb{Z}_2^k$ for some $k\in \mathbb{N}$.
\end{proposition}
\begin{proof}
    Using Lemma~\ref{lem:normalizedShifting2}, we can in $\mathcal{O}(\ell m)$ time find a spanning tree $T$ of $G$ and a $T$-normalized shifting $\gamma'$ of $\gamma$ of difference at most $\ell$ from $\gamma$. If $\langle G, \gamma'\rangle$ is isomorphic to $\mathbb{Z}_2^k$ for some $k\in \mathbb{N}$, then output ``yes'' together with the witness $\gamma'$. Otherwise, output ``no''. This can be checked in $\mathcal{O}(\ell m)\cdot \phi(12\ell)$ time by checking whether each element of $\{\gamma'(\vec{e}) : e \in E(G)\}$ has order 2 and commutes with every other element. (Note that $\ell \geq m$ and that the word length of $\gamma'(\vec{e})$ is at most $3\ell$ for each edge $e$. Moreover, we can determine whether elements $\alpha$ and $\beta$ commuting by checking whether $\alpha\beta \alpha^{-1}\beta^{-1}$ is the identity.) The correctness of this algorithm follows from Proposition~\ref{prop:3edgeConn} and Lemma~\ref{lem:normalizedShifting3}.
\end{proof}

We now give the full theorem for general group-labeled graphs using Theorem~\ref{thm:fullTechnical}.

\begin{theorem}
\label{thm:algorithm-general}
    Let $(G,\gamma)$ be a connected group-labeled graph with $m$ edges and total word length $\ell$. Then it can be decided in $\mathcal{O}(\ell m)\cdot \phi(12\ell)$ time whether every Eulerian trail from a vertex $a$ to a vertex $b$ has the same label. Moreover, in the affirmative case, the algorithm finds a shifting $\gamma'$ of $\gamma$ of difference at most $\ell$ such that Condition~\ref{itm:cond2} of Theorem~\ref{thm:fullTechnical} holds.
\end{theorem}
\begin{proof}
    We first test whether there exists any Eulerian trail from $a$ to $b$. This can be done in $\mathcal{O}(m)$ time; it suffices to check whether every vertex besides $a$ and $b$ has even degree, and if $a$ and $b$ are different, whether they have odd degree. If no such trail exists then we return that every trail vacuously has the same label. Otherwise we find such a trail $C$, which can be done in $\mathcal{O}(m)$ time by Hierholzer's algorithm~\cite{hierholzer1873moglichkeit}.
    
    We then find a partition of $V(G)$ into cores. This can be done in $\mathcal{O}(m)$ time by~\cite{NagamochiIbarakiCore}. Then, for each core $X$, we obtain a valid instance $(H,a',b')$ from the instance $(G,a,b,C)$ and a labeling $\gamma_H$ from $\gamma$ as in Section~\ref{sec:general}. This can be done in $\mathcal{O}(\ell)$ time for any particular core, and as there are at most $n$ cores, this contributes $\mathcal{O}(\ell n)$ to the running time. Note that the total word length of $(H, \gamma_H)$ is also at most $\ell$. Then, for each core, we use the algorithm from Proposition~\ref{prop:algorithm3-edge-conn} to test in time $\mathcal{O}(\ell|E(H)|)\cdot \phi(12\ell)$ whether every Eulerian trail from $a'$ to $b'$ has the same label. If so, we also obtain a shifting $\gamma'_H$ of $\gamma_H$ of difference at most $\ell$ from $\gamma_H$ such that $\langle H, \gamma'_H\rangle \cong \mathbb{Z}_2^k$ for some~$k$. Note that the sum over all these different graphs $H$ of the number of edges of $H$ equals the number of edges $m$ of $G$. Thus this step can be completed in $\mathcal{O}(\ell m)\cdot \phi(12\ell)$ time.
    
    If every valid instance $(H,a',b')$ has such a shifting, then we output ``yes'', otherwise we output ``no''. The correctness of the algorithm follows from Theorem~\ref{thm:fullTechnical}. In the event that we output ``yes'', we obtain the shifting $\gamma'$ of $\gamma$ by shifting as in $\gamma'_H$ for each~$H$.
\end{proof}

Finally, we use Theorem~\ref{thm:algorithm-general} to find trails of different labels if they exist.

\begin{corollary}
\label{cor:findTrails}
    Let $(G,\gamma)$ be a connected group-labeled graph with $m$ edges and total word length $\ell$. Let $a$ and $b$ be vertices of $G$ such that there exist two Eulerian trails from $a$ to $b$ with different labels. Then we can find in $\mathcal{O}(\ell m^3) \cdot \phi(12\ell)$ time a circuit $L$ of $G$ with $\gamma(L) \neq \gamma(L)^{-1}$ and an Eulerian trail $T$ of $G$ from $a$ to $b$ which contains $L$ as a subcircuit.
\end{corollary}
\begin{proof}
    By Theorem~\ref{thm:fullTechnical}, such a circuit $L$ exists if and only if there exist two Eulerian trails from $a$ to $b$ with different labels. Using this equivalence, we proceed as follows.

    While there exists a pair of arcs $\vec{e}$ and $\vec{f}$ such that $\vec{e} \vec{f}$ is a trail and splitting off $\vec{e}$ and $\vec{f}$ does not change whether such an $L$ exists, we split them off and label the new arc by $\gamma(\vec{e})\gamma(\vec{f})$. We can test for the existence of such an $L$ by Theorem~\ref{thm:algorithm-general}. For each arc in the graph we store a trail, starting with the trail consisting of only that arc. Each time we split off a pair of arcs, we append their corresponding trails in the correct order, and combine their labels. In this way, every arc of the resulting graph corresponds to a trail in the original graph of the same label. Arcs $\vec{e}$ and $\vec{e}^{-1}$ correspond to a trail $T$ and its inverse~$T^{-1}$.

    Note that if a pair of arcs cannot be split off in one stage of the algorithm, then they also cannot be split off in any later stage of the algorithm. To be more precise, if we determine at one stage of the algorithm that arcs $\vec{e}$ and $\vec{f}$ cannot be split off, then at any later point in the algorithm, if $\vec{e}_1$ is an arc whose corresponding trail ends with $\vec{e}$, and $\vec{f}_1$ is an arc whose corresponding trail begins with $\vec{f}$, then $\vec{e}_1$ and $\vec{f}_1$ cannot be split off. So, to speed up the algorithm, each arc $\vec{e}$ stores a list of ``forbidden'' arcs $F(\vec{e})$ so that we know $\vec{e}$ cannot be split off with any arc in $F(\vec{e})$. We update these lists every time we test a pair of arcs.

    Each application of Theorem~\ref{thm:algorithm-general} takes $\mathcal{O}(\ell m) \cdot \phi(12\ell)$ time. Updating the lists $F(\vec{e})$ can (naively) be done in $\mathcal{O}(\ell m)$ time since $\ell \geq m$. So, since throughout the course of the algorithm we only check $\mathcal{O}(m^2)$ pairs of arcs, the total runtime is $\mathcal{O}(\ell m^3) \cdot \phi(12\ell)$. 

    Once the algorithm determines that no more arcs can be split off, we claim that there are at most three edges remaining. To see this, note that there exists an Eulerian trail $T_1LT_2$ from $a$ to $b$ such that $L$ is a circuit with $\gamma(L) \neq \gamma(L)^{-1}$. If $L$ is not a loop, then a pair of arcs internal to $L$ can be split off. Similarly if $T_1$ or $T_2$ contains at least two arcs, then a pair of arcs can be split off. Thus the resulting graph has at most three edges and $L$ is a loop. We then recover $T_1, L, T_2$ in the original graph by reading off the trails corresponding to $T_1,L,T_2$ in the small graph.
\end{proof}

Together, Theorem~\ref{thm:algorithm-general} and Corollary~\ref{cor:findTrails} immediately imply Theorem~\ref{thm:algorithmicIntro} from the introduction.

\bibliographystyle{abbrv}
\bibliography{eulerian}

\end{document}